\documentclass{amsart}
\usepackage{amsfonts}
\usepackage{amsmath}
\usepackage{amscd}
\usepackage{amssymb}
\usepackage{amsxtra}
\usepackage{latexsym}
\usepackage{amsthm}
\usepackage{mathabx}
\usepackage{graphicx}
\usepackage[bookmarks,colorlinks,pagebackref]{hyperref} 
\input{xy}
\xyoption{all}

\theoremstyle{plain}

\newtheorem*{conjectuur*}{Conjecture}

\swapnumbers
\newtheorem{theorem}[subsection]{Theorem}
\newcommand\Thm[1]{Theorem~\ref{#1}}
\newtheorem{corollary}[subsection]{Corollary}
\newcommand\Cor[1]{Corollary~\ref{#1}}
\newtheorem{lemma}[subsection]{Lemma}
\newcommand\Lem[1]{Lemma~\ref{#1}}
\newtheorem{proposition}[subsection]{Proposition}
\newcommand\Prop[1]{Proposition~\ref{#1}}

\newtheorem{conjecture}[subsection]{Conjecture}
\newcommand\Conj[1]{Conjecture~\ref{#1}}

\newtheorem*{citedtheorem}{Theorem}

\theoremstyle{definition}
\newtheorem{definition}[subsection]{Definition}

\newtheorem{example}[subsection]{Example}
\newcommand\Examp[1]{Example~\ref{#1}}

\theoremstyle{remark}

\newtheorem{remark}[subsection]{Remark}
\newcommand\Rem[1]{Remark~\ref{#1}}

\swapnumbers

\newcommand{\emptyprop}{q}

\newcommand \seq[2]{#1\mathstrut_{#2}}
\newcommand \ul[1]{\seq{#1}\natural }
\newcommand\cp[1]{{#1}\mathstrut_\sharp}

\newcommand \acf{algebraically closed field}

\newcommand \ann[2]{\operatorname{Ann}_{#1}(#2)}

\newcommand \ch{characteristic}
\newcommand \CM{Coh\-en-Mac\-au\-lay}
 
\newcommand \DVR{discrete valuation ring}

\newcommand \homo{homomorphism}

\renewcommand\iff{if and only if}

\newcommand \inv[1]{{#1^{-1}}}

\newcommand \iso{\cong}
\newcommand \loc{{\mathcal {O}}}
\newcommand \los{\L os' Theorem}
\newcommand \map[1]{{\newcommand{\tmpprop}{#1q}  \if\tmpprop\emptyprop \to\else \xrightarrow{{\phantom{i}{#1}\phantom{i}}}\fi}} 
\newcommand \maxim{\mathfrak m}
\newcommand \nat{\mathbb N}
\newcommand \norm[1]{\left|#1\right|}
\newcommand \onto{\twoheadrightarrow}
\newcommand \op\operatorname
\newcommand \pol[2]{#1[#2]}
\newcommand \pow[2]{#1[[#2]]}

\newcommand \pr{\mathfrak p}
\newcommand \primary{\mathfrak g}
\newcommand \range [2]{#1,\dots,#2}

\newcommand \rij[2]{(#1_1,\dots,#1_{#2})}
\let\sub\subseteq

\newcommand \tensor{\otimes}
\newcommand \tor[4]{\operatorname{Tor}^{#1}_{#2}(#3,#4)}

\newcommand \zet{\mathbb Z}

\newcommand \exactseq [5]{0\to{#1}\:\map{#2}\:{#3}\:\map{#4}\:{#5}\to0}

\newcommand{\commdiagram}[9][]{%
\begin{equation}
{\newcommand{\tmpprop}{#1q} 
\if\tmpprop\emptyprop \relax\else \label{#1}\fi}
\begin{aligned}%
\mbox{
\begin{picture}(130,90)%
\put(120,70){\vector( 0,-1){50}}%
\put(10,80){\vector( 1, 0){100}}%
\put(0,70){\vector( 0,-1){50}}%
\put(10,10){\vector( 1, 0){100}}%
\put(115,80){\makebox(0,0)[l]{$#4$}}%
\put(5,80){\makebox(0,0)[r]{$#2$}}%
\put(115,10){\makebox(0,0)[l]{$#9$}}%
\put(5,10){\makebox(0,0)[r]{$#7$}}%
\put(-3,50){\makebox(0,0)[r]{$#5$}}
\put(123,50){\makebox(0,0)[l]{$#6$}}
\put(60,3){\makebox(0,0)[c]{$#8$}}
\put(60,88){\makebox(0,0)[c]{$#3$}}
\end{picture}}
\end{aligned}
\end{equation}}

\newcommand\commtrianglefront[7][]{%
\begin{equation}
{\newcommand{\tmpprop}{#1q} 
\if\tmpprop\emptyprop \relax\else \label{#1}\fi}
\begin{aligned}%
\mbox{
\begin{picture}(120,80)%
\put(55,68){\vector(-1,-2){30}}
\put(65,68){\vector(1,-2){30}}
\put(30,5){\vector(1,0){60}}
\put(60,75){\makebox(0,0)[c]{$#2$}}
\put(25,5){\makebox(0,0)[r]{$#4$}}
\put(95,5){\makebox(0,0)[l]{$#6$}}
\put(60,0){\makebox(0,0)[c]{$#5$}}
\put(37,43){\makebox(0,0)[r]{$#3$}}
\put(83,43){\makebox(0,0)[l]{$#7$}}
\end{picture}}
\end{aligned}
\end{equation}}

\newcommand\commtriangleback[7][]{%
\begin{equation}
{\newcommand{\tmpprop}{#1q}
\if\tmpprop\emptyprop \relax\else \label{#1}\fi}
\begin{aligned}%
\mbox{
\begin{picture}(120,80)%
\put(55,70){\vector(-1,-2){30}}
\put(65,70){\vector(1,-2){30}}
\put(30,5){\vector(1,0){60}}
\put(60,75){\makebox(0,0)[c]{$#2$}}
\put(25,5){\makebox(0,0)[r]{$#6$}}
\put(95,5){\makebox(0,0)[l]{$#4$}}
\put(60,0){\makebox(0,0)[c]{$#5$}}
\put(37,43){\makebox(0,0)[r]{$#7$}}
\put(83,43){\makebox(0,0)[l]{$#3$}}
\end{picture}}
\end{aligned}
\end{equation}}

 \newcommand\Wlike{Witt-closed}
\newcommand \powint[1]{#1^{\text{pow}}}
\newcommand \wi[1]{\mathcal W(#1)}
\newcommand \pint[2]{\frob{#1}^{\text{int}}(#2)}
\newcommand \sat[3]{\op{sat}_{#1}(#2;#3)}
 \newcommand\qdef{quasi-deformation}
 \newcommand\vs{very small}
 \newcommand\fr{Frobenius transform}
\newcommand \Witt[1]{#1^{\mathcal W}}
\newcommand \scal[2]{#1_{#2}^\wedge}
 \newcommand\an{analytic}
  \newcommand\trace[2]{\op{Tr}_{#1}({#2})}
\newcommand\mat[1]{\mathbb{#1}}
\newcommand\frob[1]{\mathbf{F}_{#1}}
\newcommand\tuple[1]{\mathbf{#1}}
\newcommand\tr[2]{#1^{(#2)}}

\title {Maximal Cohen-Macaulay modules over local toric rings}
\author{Hans Schoutens}
\date\today
\address{Department of Mathematics\\
NYC College of Technology and
the CUNY Graduate Center\\
New York, NY, USA}
\email{hschoutens@citytech.cuny.edu}
\begin{document}
\begin{abstract} 
In analogy with the classical, affine toric rings, we define a local toric ring as the quotient of a   regular local ring modulo an ideal generated by binomials in a regular system of parameters with unit coefficients; if the coefficients are just $\pm1$, we call the ring purely toric. 
  We prove  the following  results  on the existence of maximal \CM\ modules: (EQUI$\mathstrut_p$)  we construct certain families of  local toric    rings    satisfying Hochster's small \CM\ conjecture in positive \ch; (EQUI$\mathstrut_0$) provided Hochster's small \CM\ conjecture holds in positive \ch\ with the additional condition that the multiplicity of the small \CM\ module is bounded in terms of the parameter degree of the ring, then any local ring (not necessarily toric) in  equal \ch\ zero admits a formally etale  extension satisfying Hochster's small \CM\ conjecture (this applies in particular to the families from (EQUI$\mathstrut_p$)); and (MIX), in mixed \ch, we show that all    purely toric local rings   satisfy Hochster's big \CM\ conjecture, and so do those belonging to the families from (EQUI$\mathstrut_p$).
\end{abstract}
\maketitle



\section{Introduction}
Recall that a (not necessarily finitely generated, aka \emph{big}) module $Q$ over a Noetherian local ring $(R,\maxim)$ is called \emph\CM, if $\dim M=\op{depth}(M)$. One is especially interested in the case where this common value is also the dimension of $R$, and we will refer to such a module as a \emph{maximal \CM} module, and abbreviate it as MCM.\footnote{Hochster and others sometimes leave out the `maximal', which might lead to confusion.} When an MCM  is   finitely generated, we      emphasize this by calling it \emph{small}.   By taking completion, one can always ensure that $Q$ satisfies the following stronger property: any system of parameters  on $R$ is $Q$-regular ($Q$ is then a so-called \emph{balanced MCM}), and we will henceforth just mean this latter, when discussing the existence of MCM's.

The interest of these notions, introduced by Hochster, is that a Noetherian local ring  admitting an MCM satisfies almost all of the homological conjectures; for Serre's positivity conjecture one needs the MCM to be moreover small (see \S\ref{s:intform} for more details). Since all homological conjectures admit faithfully flat descent, there is no loss of generality in proving the existence of (big or small) MCM's after taking a scalar extension (in the sense of \cite[\S3]{SchClassSing}), and so we may assume that $R$ is furthermore complete and has algebraically closed residue field, which we may even take to have uncountable cardinality. Moreover, in the mixed \ch\ case, to avoid pathologies, we require that $R$ is torsion-free over some complete $p$-ring (=mixed \ch\ complete \DVR\ with uniformizing parameter $p$). To not have to repeat all these conditions in this introduction, we call such a ring temporarily  \emph{\an}.
The current state of the existence of MCM's for   \an\ ring{}s, leaving aside some   special cases, can be summed up  as follows:

\begin{citedtheorem}[\cite{HoBCM3,HHbigCM,HoCurr}]\label{T:knownMCM}
A $d$-dimensional  \an\ ring  admits 
\begin{description}
\item[Hochster-Huneke] in equal \ch,   a big MCM \emph{algebra};
\item[Hochster/Heitmann] in mixed \ch\ and $d\leq 3$,   a big MCM;
\item[Hochster/Hartshorne/Peskine-Szpiro] if $d=2$, or, in the graded case,  if $d=3$ and    the \ch\ is positive,     a small MCM.
\end{description}
\end{citedtheorem} 

So the two major open problems are: ($\text H_{\text{big}}$) the existence of big MCM's in mixed \ch\ in dimensions four and higher; and ($\text H_{\text{small}}$) the existence of small MCM's, in any \ch,  in dimension three and higher.\footnote{The one known three-dimensional  (graded, positive \ch) is really the case of a cone over  a two-dimensional variety; likewise, the  few known higher-dimensional cases are somehow derived from lower dimensional cases, like \cite{HanCM}.} The most pressing of these is   ($\text H_{\text{big}}$), as many homological conjectures are still outstanding in mixed \ch. On the other hand, Hochster seems skeptical about ($\text H_{\text{small}}$) if $R$ does not have positive \ch. We will formulate a slightly stronger version of ($\text H_{\text{small}}$) and dissuade Hochster's distrust in equal \ch\ zero. Namely, let us say that an MCM is \emph\vs\  if its multiplicity is at most the equi-parameter degree of $R$. Recall that the \emph{parameter degree} $\op{pardeg}(R)$ of $R$ is the smallest possible length of $R/I$ where $I$ varies over all parameter ideals (=ideals generated by a system of parameters); in mixed \ch\ $p$, we must use instead the \emph{equi-parameter degree} $\op{epardeg}(R)$, which is the  smallest possible length of $R/I$ where $I$ varies over all parameter ideals containing $p$. Since the parameter degree of $R$ is equal to its multiplicity \iff\ $R$ is \CM,  local \CM\ rings trivially admit a \vs\ MCM, and I now conjecture more generally:

\begin{conjecture}\label{C:ssMCM}
Any complete local ring admits a \vs\ MCM.
\end{conjecture} 

It is the latter form of ($\text H_{\text{small}}$) that can be lifted to equal \ch\ zero, by an ultraproduct argument (which requires the residue field to be uncountable; see  \S\ref{s:mix}):
 
\begin{theorem}\label{T:Hochequi}
If every $d$-dimensional  \an\ ring  of positive \ch\ admits a \vs\ MCM, then so does any   $d$-dimensional    \an\ ring  of equal \ch.
\end{theorem}

So, we should focus on problem ($\text H_{\text{small}}$) for \an\ ring{s} of positive \ch, under the additional \vs{}ness condition, and we will now describe a class of \an\ ring{s} for which we can show the validity of the conjecture.

 Recall that an \emph{affine toric variety} over a field $\kappa$  is an irreducible variety containing a torus $(\kappa^*)^n$ as a Zariski open subset such that the action of the torus on itself extends to an algebraic action on the whole variety. Hochster \cite{HoToric} proved that a normal toric variety is automatically \CM. Since the integral closure of a toric variety is again toric, it follows that
 any local ring  of a toric variety admits a small MCM. To generalize this, we use the following algebraic characterization of toric varieties (see, for example, \cite[\S1.1]{CLS}):    the ideals defining affine toric varieties are precisely the prime ideals generated by pure binomials (=polynomials in the variables $\tuple y:=\rij yn$ of the form $\tuple y^\alpha\pm\tuple y^\beta$ with $\alpha,\beta\in\nat^n$).\footnote{Following common practice, we write $\tuple y^\alpha$  for $y_1^{a_1}\cdots y_n^{a_n}$, for $\alpha=\rij an$.}
  In \cite{EisStu}, this was taken as the departing point of investigating general \emph{toric} or \emph{binomial} ideals in a   polynomial   ring over a field as those ideals generated by binomials (with arbitrary non-zero coefficients, not just $\pm1$). 
We consider the following local version: 
\begin{definition}\label{D:tor}
Given an unramified\footnote{Complete ramified regular local rings might not be power series rings, a technical complication I prefer to avoid, although I do not know how essential it is.} regular local ring $S$ and a regular system of parameters $\tuple y$ in $S$, call an ideal $I\sub S$ \emph{toric} (with respect  to the given regular system of parameters), if it is generated by `binomials' of the form $\tuple y^\alpha-u\tuple y^\beta$, where $u\in S^*$ is a unit in $S$. More generally, if $\Xi\sub S^*$ is a subgroup, then we call the ideal \emph{$\Xi$-toric}, if all coefficients $u$ belong to $\Xi$. In particular, if $\Xi=\{\pm1\}$,  we call the ideal \emph{purely toric}. The corresponding quotient $S/I$ will be called a ($\Xi$-)\emph{toric local ring}.
\end{definition} 
%
%
  Since we work locally, there are  many more units, and so we get a much larger class than in the affine case (e.g., not only the cusp $0=x^2-y^3$ but also the node $0=x^2-y^2-y^3=x^2-(1+y)y^2$ is toric at the origin).  We are mainly interested in \an\ toric rings, that is to say,  when $S$ is a power series ring, either over   an  uncountable \acf\ or over a complete $p$-ring, and so we would like to   prove that any  toric \an\ ring admits  a small MCM.
 
\subsection*{Mixed \ch} We give the following   positive solution of ($\text H_{\text{big}}$):

\begin{theorem}\label{T:mixtorMCM}
Every purely toric \an\ ring in mixed \ch\ admits a big MCM.
\end{theorem} 

We prove in fact a more general theorem (\Thm{T:Wittbcm}): using Witt vectors, we define the \emph{Witt transform} $\Witt R$ of an \an\ ring $R$; it is again  \an, of the same dimension, and we call $R$ \emph\Wlike, if $R\iso \Witt R$. We then show that any \Wlike\ local ring admits a big MCM. The theorem follows since the multiplicativity of the Teichmuller character implies that   purely toric rings are \Wlike. There is a more general class of toric rings for which we can prove \Thm{T:mixtorMCM}, which we will discuss below. It should be pointed out though that our construction is   different from the one in \cite{HoWitt}.

\subsection*{Equal \ch}   Fix an uncountable \acf\ $\kappa$ of \ch\ $p\geq 0$. The  $\kappa^*$-toric case can be reduced to  the affine case, using Hochster's result (where $\kappa^*$ is the multiplicative group of $\kappa$).
However, we will show that  for a toric domain $R$ with $p>0$, even in the affine case, there is a  candidate for an MCM which is usually smaller  than its normalization---in particular, it is \vs---,  namely the ring $\pint{}R$ of \emph{F-integral elements} (=the elements in its field of fraction such that some $p^n$-th  power   lies in $R$); see \Conj{C:afftorpowint} for a \ch-free variant.
%
   To have such `smaller' MCM's might be advantageous for  calculating intersection forms, as we will briefly explain in \S\ref{s:intform}. Although we cannot prove the result in general, we will prove it for certain families of   toric \an\ rings.
%
More concretely, consider the following families $\mathcal T_{d,n,m}(V)$ of  \an\ ring{s}: let  $V$ be     either  a field or a complete $p$-ring, and let $\mathcal T_{d,n,m}(V)$ consist of all $R:=\pow S{\tuple u}/I$, where $S:=\pow V {y_1,\dots,y_d}$ and  $\tuple u=\rij un$, where $\tuple y$ is either $\rij yd$ (equal \ch\ case) or $(p,y_1,\dots,y_d)$ (mixed \ch\ case), and where $I$ is a  toric ideal generated by
\begin{description}
\item[($n\geq 2$,  $m=2$)]   binomials of the form $u_i^2-a_i\tuple y^{\alpha_i}$ and $u_iu_j-b_{ij}\tuple y^{\beta_{ij}}$, with $a_i,b_i$  units in $S$ such that $a_ia_j=b_{ij}^2$   and $\alpha_i+\alpha_j=2\beta_{ij}$;
\item[($n=2$, $m$ odd)]     binomials of the form $u_i^m-a_i\tuple y^{\alpha_i}$, for $i=1,2$, and $u_1u_2-b\tuple y^\beta$, $a_1, a_2,b$  units in $S$ such that $a_1a_2=b^m$  and $\alpha_1+\alpha_2=m\beta$.
\end{description}

Note that the conditions on the $a_i$ and $\alpha_i$ guarantee that in either case the ideal $I$ has height $n$ (whence $\dim R=d+\dim V$) and $S\sub R$ is finite,  yielding a `toric' Noether normalization. 
We prove the following:

\begin{theorem}\label{T:dne}
For $R$ an \an\ ring belonging to  $\mathcal T_{d,n,m}(V)$,
\begin{enumerate}
\item[(EQUI)] if $V$ is an uncountable \acf\ whose \ch\ does not divide $m$, then $R$ admits a \vs\ MCM;
\item[(MIX)] if $V$ is a complete $p$-ring, then $R$ admits a big MCM.\footnote{Although these are classes of non-purely toric ideals, in mixed \ch, \Thm{T:mixtorMCM} also applies to them by a specialization argument, \Thm{T:torNNmix}.}
\end{enumerate}
\end{theorem}

For the remainder of the introduction, I will now explain the first case, when $V=\kappa$ is an uncountable \acf, and by a variant of \Thm{T:Hochequi} (see \Rem{R:Hochequi}), we only need to consider the case that the \ch\ $p$ is positive. The rings $R$ in $\mathcal T_{d,n,m}(\kappa)$ are not domains, but they all admit a $d$-dimensional, minimal prime ideal $\pr$ which is  toric. We will show that the F-integral elements of $R/\pr$ form a small MCM over $R/\pr$ whence over $R$. (For a domain $A$ of \ch\ $p$, we call an element $f$ in its field of fractions \emph{F-integral}, if $f^q\in A$ for some $q=p^n$.) However, it is useful to have a construction that does not require the ring to be a domain, and for that, we have to look at 
%
%
the \emph{\fr} of $R$: viewing $R$ as a module over itself via the Frobenius $\frob p\colon R\to R$, we will denote it as $\frob*R$; likewise, elements of $\frob*R$ will be denoted by $*r$, so that the scalar action is given by $s{*}r=*s^pr$. Our desired MCM  will be an $R$-submodule $Q$ of $\frob *R$  (where for small  $p$ we may have to take instead some iterate of $\frob p$). Concretely, $Q$ is the $S$-saturation of the (cyclic) $R$-submodule   generated by $*1$: namely,  the $S$-submodule of $\frob*R$ consisting of all $*r$ such that $s{*}r=a{*}1$ for some non-zero $s\in S$ and $a\in R$. It follows that $Q$ is again an $R$-module, but the surprising fact is that it is free as an $S$-module (of rank $m$), whence a \vs\ MCM (the proof of \Thm{T:Hochequi} then automatically also proves the equal \ch\ zero case). However, this freeness result is still an enigma: I can write down   explicitly $m$ generators over $S$ (which are easily seen to be linearly independent over $S$), but I do not know how to find these generators for more complicated toric ideals. At the moment, everything relies on some basic congruence (arithmetic) relations, but it is not clear what these would become in a more general setting. 

The toric rings in the families $\mathcal T$ are in fact special cases of a more general construction of toric rings: recall that in the affine case, a toric ideal can also be defined in terms of a (partial) character defined on a sublattice $\Gamma\sub\zet^n$ (\cite[Theorem 2.1]{EisStu}). We will consider below the case where $\Gamma$ is moreover the graph of a morphism (the resulting binomial equations are then `bi-partite', that is to say, of the form $\tuple x^\alpha-u\tuple y^\beta$, with $(\tuple x,\tuple y)$ a fixed partition of the variables). Although I cannot show (EQUI) in \Thm{T:dne}  for this larger class of rings, I give an example where we verify the conjecture. On the other hand, we do verify  (MIX) in general for this class. 

A note of caution: we cannot expect that for non-toric \an\ domains, their ring of F-integral elements is always MCM: Bhatt \cite{BhNonCM} has given examples  in positive \ch\footnote{In equal \ch, normal domains split off in any finite extension, and so does their local cohomology, so no finite extension of a non-\CM\ normal domain can be \CM.} of normal \an\ domains $R$ that do not admit a small MCM which in addition carries the structure of an $R$-algebra. In a future paper, we will give a new condition (involving local cohomology) for the existence of a small MCM and conjecture that this condition can be realized inside the \fr\ $\frob*M$ of some, possibly non-free, $R$-module $M$.


\section{Hochster's  MCM conjectures and Hochster rings}
Recall that a \emph{scalar extension} of local rings $(R,\maxim)\to (S,\mathfrak n)$ is a formally etale \homo, that is to say, faithfully flat and  unramified (meaning that $\maxim S=\mathfrak n)$. In particular, $R$ and $S$ have the same dimension and depth. Moreover, if $\kappa$ is the residue field of $R$ and $\lambda$ any field extension of $\kappa$, then there exists a unique scalar extension $\scal R{\lambda}$ of $R$ which is complete and has residue field $\lambda$ (see,  for instance, \cite[\S3]{SchClassSing}). The philosophy for  allowing scalar extensions is that they do not change the singularities and that any property that descends under faithfully flat maps is inherited from the scalar extension.

Let us therefore for call a local ring $R$   \emph{Hochster}, if some scalar extension of $R$ admits a small  MCM, that is to say, a finitely generated module of   depth equal to the dimension of $R$. It is well-known that any  two-dimensional local ring is Hochster. By work of Hartshorne, Peskine-Szpiro, and Hochster, any three-dimensional graded domain in positive \ch\  is Hochster (see, for instance \cite[Corollary 2]{HoCurr}). We will call $R$ \emph{strongly Hochster}, if some scalar extension of $R$ admits a very small MCM. Given two local rings $S$ and $R$, let us say that $R$ is a \emph\qdef\ of $S$, if there is a diagram  
\begin{equation}\label{eq:qdef}
\xymatrix{
&S\ar@{->>}[d]\\
R\ar@{^{(}->}[r]^g&\bar S\\
}
\end{equation} 
with $g$ finite and injective and  $\bar S=S/\rij xnS$ with $\rij xn$ part of a system of parameters in  $S$. 
We sometimes refer to the whole diagram~\eqref{eq:qdef} as a \qdef. The completion of a   \qdef\ is again a  \qdef, as is any flat base change.

\begin{lemma}\label{L:qdefHoc}
Let $R$ be a \qdef\ of $S$. 
If $S$ admits a big MCM, then so does $R$. If $S$ is moreover Hochster, then so is  $R$.
\end{lemma} 
\begin{proof}
In either case,  we may replace $R$ by its completion and   assume form the start that we have a \qdef~\eqref{eq:qdef} with $R$  and $S$ complete. Let $\kappa$ and $\lambda$ be the respective residue fields of $R$ and $S$. Let $Q$ be an MCM  module over $S$, which we may assume to be balanced, and put $\bar Q:=Q/\rij xnQ$. Let $\rij yd$ be a system of parameters of $R$, which therefore is also   a system of parameters in $\bar S$. Let us continue to write $y_i$ for a choice of lifting in $S$. Hence $(x_1,\dots,x_n,y_1,\dots,y_d)$ is a system of parameters in $S$, whence is $Q$-regular. It follows that $\rij yd$ is $\bar Q$-regular, showing that $\bar Q$ is an MCM over $R$. If $S$ is moreover Hochster, then there exists a scalar extension $S\to S'$ such that $Q$ is a finite $S'$-module. We may replace $S'$ and $Q$ by their completion, and hence $S'\iso \scal S{\lambda'}$ by \cite[Proposition 3.4]{SchClassSing}, where $\lambda'$ is the residue field of $S'$. Therefore,   $R':=\scal R{\lambda'}$  is a scalar extension of $R$ and  $\bar Q$ is finitely generated as an $R'$-module.
\end{proof} 

Hence, in view of \Conj{C:FRMCM}, below, we may ask which complete local rings are \qdef{s} of  toric local  rings. Bhatt's counterexamples to the existence of small MCM algebras (\cite{BhNonCM}) shows that, if we believe in the conjecture, then not every local ring can be obtained this way.

\section{Calculating  intersection  forms using small MCM's}\label{s:intform}
As mentioned above, the existence of an MCM has major implications for the homological theory of a ring. In particular, many homological conjectures remain open in mixed \ch\ since we do not know yet the existence of MCM's in that case. For most applications, it suffices to find a big MCM, but there is at least one case where this is not enough: to prove the positivity conjecture of Serre for ramified regular local rings in mixed \ch\ (the remaining open case). To this end,  it suffices by faithfully flat descent to show  that local rings in mixed \ch\ are Hochster. But even in equal \ch\ it is useful to show that local rings are Hochster:  although the positivity conjecture is known in that case,   small MCM's can aid in calculating the intersection form, as I will now briefly explain. 

Let $X$ be a   variety, $Y,Z\sub X$ closed subvarieties, and $x$ a point in their intersection $Y\cap Z$. Suppose $x$ is a regular point on $X$, and the intersection at $x$ is \emph{in general position}, in the sense that the sum of the local dimensions of $Y$ and $Z$ at $x$ is at most the local dimension of $X$ at $x$ and $Y\cap Z$ is zero-dimensional at $x$. The latter means that $\loc_{Y,x}\tensor\loc_{Z,x}$ has finite length, and this finite length  $\ell(\loc_{Y,x}\tensor\loc_{Z,x})$ is called  the \emph{`naive' intersection multiplicity} of $Y$ and $Z$ at $x$. If $Y$ and $Z$ are curves in the plane, this leads to the correct Bezout formula, meaning that the sum over all intersections points is the product of the degrees of the curves. However, Serre realized that this is no longer true in higher dimensions due to the lack of \CM{ness}, and he suggested the following  \emph{local intersection multiplicity}:
\begin{equation}\label{eq:Serre}
\chi(Y,Z;x):=\sum_{i=0}^\infty (-1)^i\ell(\tor {\loc_{X,x}}i{\loc_{Y,x}}{\loc_{Z,x}})
\end{equation} 
Note that the sum is well-defined since all higher Tor's also must have finite length and vanish for $i$ bigger than the dimension of $X$. We need a definition that includes also modules, and for this it is easier to formulate everything in terms of local rings. Hence, 
let $(S,\mathfrak n)$ be a  $d$-dimensional regular local,
let $M$ and $N$ be $S$-modules such that $M\tensor N$ has finite length, meaning that $\ann {}M+\ann{}N$ is $\mathfrak n$-primary, and $\dim M+\dim N\leq d$. We then define their intersection form similarly as
\begin{equation}\label{eq:Serremod}
\chi(M,N):=\sum_{i=0}^d \tor SiMN.
\end{equation} 
Serre proved that  $\chi(M,N)\geq 0$, with  equality \iff\ $\dim M+\dim N<d$, whenever $S$ is equi\ch\ or unramified. To discuss the use of small MCM's, we need a vanishing theorem.


\begin{lemma}\label{L:Torvan}
Let $R$ be a $d$-dimensional local ring and let $M$ and $N$ be $R$-modules such that $M\tensor N$ has finite length. Then $\tor RiMN=0$ for all $i>\op{pd}(M)+\op{pd}(N)-d$.
\end{lemma} 
\begin{proof}
There is nothing to prove if either module has infinite projective dimension, so assume the modules have moreover finite projective dimension. Let $i$ be maximal such that there exists a pair  of $R$-modules $M$ and $N$  of finite projective dimension $p$ and $q$ respectively, such that $M\tensor N$ has finite length  and $\tor RiMN\neq 0$. We need to show that $i\leq p+q-d$, so suppose not, that is to say, $i>p+q-d$.   
  Note that for $i>q$ we automatically have $\tor RiMN=0$, showing that necessarily $p<d$. By   Auslander-Buchsbaum therefore, $M$ must have positive depth (and likewise $N$). Let $x$ be an $M$-regular element. Since $M\tensor N$ has finite  length, some power $x^n$ lies in $\ann{}M+\ann{}N$, say, $x^n=a+b$ with $aM=0$ and $bN=0$. Hence replacing $x$ by $b$, we may moreover assume that $xN=0$. We have an exact sequence $\exactseq MxM{}{M/xM}$, which tensoring with $N$ gives part of a long exact sequence
\begin{equation}\label{eq:torMb}
\tor R{i+1}{M/xM}N\to\tor RiMN\map x\tor RiMN.
\end{equation} 
By Auslander-Buchsbaum, $\op{pd}(M/xM)=p+1$, and so $i+1>p+1+q-d$. Since $(M/xM)\tensor N$ too has finite length,  maximality of $i$   yields $\tor R{i+1}{M/xM}N=0$. On the other hand, $x$ is zero on $N$ whence on $\tor RiMN$, so that  \eqref{eq:torMb} implies that also $\tor RiMN=0$, contradiction.
\end{proof} 

MCM's come in via the following corollary, which implies that the `naive intersection form' is the correct one in the \CM\ case:

\begin{corollary}\label{C:intCM}
Let $S$ be a regular local ring of dimension $d$, let $M$ and $N$ be \CM\ $S$-modules such that $M\tensor N$ has finite length and $\dim M+\dim N=d$. Then $\chi(M,N)=\ell(M\tensor N)>0$.
\end{corollary} 
\begin{proof}
Let $p$ and $q$ be the respective dimensions, whence depths of $M$ and $N$, so that $d=p+q$ by assumption. By Auslander-Buchsbaum, $\op{pd}(M)=d-p$ and $\op{pd}(N)=d-q$, and hence $\tor RiMN=0$ by \Lem{L:Torvan}, for all $i>d-p+d-q-d=0$.
\end{proof} 

We can now describe an algorithm for calculating this intersection form, in case we have explicit, small MCM's: let $\pr$ and $\mathfrak q$ be prime ideals in $S$ whose sum is $\mathfrak n$-primary, with respective dimensions $p$ and $q=d-p$.  Let $M$ be a small MCM over $S/\pr$, and similarly $N$ a small MCM over $S/\mathfrak q$. Consider a prime filtration $M_s\sub\dots\sub M_2\sub M$ of $M$, meaning that each subsequent quotient $M_i/M_{i+1}$ is of the form $S/\primary_i$ with each $\primary_i$ some prime ideal in the support of $M$. It is well-known that $\pr$ appears exactly $a:= \ell(M_\pr)$ times among the $\primary_i$. Since   $\chi(\cdot,N)$ is additive, and since it vanishes on modules of dimension less than $d-q=p$, whence in particular on  the $S/\primary_i$ with $\primary_i\neq\pr$, we get $\chi(M,N)=\sum\chi(S/\primary_i,N)=a\chi(S/\pr,N)$. Similarly, with $b:=\ell(N_{\mathfrak q})$, the same argument then yields that
$$
\chi (M,N)=a\chi(S/\pr,N)=ab\chi(S/\pr,S/\mathfrak q)
$$
By \Cor{C:intCM}, however, we know that $\chi(M,N)=\ell(M\tensor N)$, so that we obtained the following formula for the intersection form of two varieties in terms of the naive intersection form of their small MCM's:
\begin{equation}\label{eq:}
\chi(S/\pr,S/\mathfrak q)=  \frac{\ell(M\tensor N)}{\ell(M_\pr)\ell(N_{\mathfrak q})}.
\end{equation} 
In particular, if one of the varieties is already \CM, say $S/\pr$, so that we may take $M=S/\pr$, the formula becomes
\begin{equation}\label{eq:multnaive}
\chi(S/\pr,S/\mathfrak q)=  \frac{\ell(N/\pr N)}{\ell(N_{\mathfrak q})}.
\end{equation} 
Suppose  $N$ is very small, so that $\ell(N_{\mathfrak q})\leq \op{mult}(N)\leq \op{epardeg}(S/\mathfrak q)$ by \cite[Theorem 6.4]{SchOrdLen} and since $N$ is generated by at most $\ell(N/\pr N)$ elements, we get a lower bound
$$
\frac{\mu(N)}{\op{epardeg}(S/\mathfrak q)}\leq \chi(S/\pr,S/\mathfrak q).
$$

\section{The ring of F-integral elements}
Let $R$ be a Noetherian domain of \ch\ $p>0$ and  $K$  its field of fractions.  

\begin{definition}\label{D:pint}
Let $q$ be a power of the \ch\ $p$. We say that an element $f\in K$ is  \emph{$q$-integral} if $f^q\in R$. One easily checks that the sum and product of $q$-integral elements is again $q$-integral, so that they form a subring $\pint qR$ of $K$, called   the \emph{ring of $q$-integral elements}. 
\end{definition} 

Clearly, $\pint qR$ lies inside   the integral closure $\tilde R$ of $R$.  If $R=\pint pR$, then we say that $R$ is \emph{F-normal}. We always have a chain of subrings $R\sub\pint pR\sub\pint {p^2}R\sub \dots\sub\tilde R$, which eventually must stabilize since $\tilde R$ is finite over $R$. This stable value, say $\pint qR$, is therefore  an F-normal ring, called the \emph{F-normalization}  of $R$ and will simply be denoted $\pint{}R$. Elements in $\pint {}R$ will simply be called \emph{F-integral}. It is not hard to see that $R$ is F-normal \iff\ every principal ideal is Frobenius closed (compare this with the fact that $R$ is normal \iff\ every principal ideal is tightly closed; see, for instance,  \cite{HuTC} for details on these closure operations).

We cannot expect $\pint {}R$ to be always a small MCM for $R$: indeed, if $R$ is normal but not \CM, then $R\sub \pint qR\sub \tilde R$ are all equal, but not \CM.\footnote{As explained in the introduction, a more serious obstruction is given by the counterexamples in \cite{BhNonCM}.}  For affine toric domains, however,    normal implies \CM, by a result of Hochster, and one could expect the same to hold for local toric domains. We can now conjecture the following sharper version in positive \ch:

\begin{conjecture}\label{C:FRMCM}
If $R$ is a complete toric domain of positive \ch, then $\pint {}R$ is \CM, whence in particular a small MCM algebra. 
\end{conjecture} 

We now give an alternative description of $\pint {}R$, which will allow us to verify the conjecture in certain cases. Moreover, the construction also works for non-domains. In fact, we are not interested in the ring structure on $\pint{}R$, but instead will realize it as a certain submodule of $\frob {*}R$. Recall that $\frob q\colon R\to R$ denotes the Frobenius morphism $x\mapsto x^q$, where $q$ is some power of the \ch\ $p$. Viewing $R$ as an $R$-module via this morphism, we get an $R$-module which we denote by $\frob{q*}R$, or just $\frob*R$, if $q$ is clear from the context; we call it the \emph\fr\ of $R$ (with respect to $q$). We will denote elements in $\frob{q*}$ by $*_qa$, or just $*a$ if $q$ is clear from the context. The $R$-module action on $\frob{q*}R$ is now conveniently given by $r{*}a:=*r^qa$. We make $\frob*R$ into an $R$-algebra, by giving it a   ring structure via $(*b)\cdot(*c):=*bc$. Hence as   abstract rings, $R$ and $\frob*R$ are isomorphic, but not as $R$-algebras. In fact, by Kunz's theorem, $\frob*R$ is flat as an $R$-module, of rank $q^d$, \iff\ $R$ is a $d$-dimensional regular  ring. 

Suppose $R$ is a domain with field of fractions $K$. We may likewise  view $\frob{q*}K$ as a $K$-algebra, whence by restriction as an $\pint qR$-algebra. One easily checks that    $\frob{q*}R$ is invariant under $\pint qR$,  so that $\frob{q*}R$ has the structure of an  $\pint qR$-algebra. In fact, we will shortly realize $\pint qR$ as an $R$-submodule of $\frob {q*}R$.

Let $R$ be a Noetherian ring, $\Sigma\sub R$ a multiplicative set in $R$,  let $M$ be an $R$-module and $N\sub M$ a submodule. We say that $N$ is \emph{$\Sigma$-saturated} in $M$, if  $sm\in N$ for some $m\in M$ and $s\in \Sigma$ implies that already $m\in N$.  This is equivalent with the canonical map $M\to \inv\Sigma(M/N)$ being injective. We define the \emph{$\Sigma$-saturation} $\sat \Sigma NM$ of $N$ in $M$ to be the submodule of all $m\in M$ such that $sm\in N$ for some $s\in \Sigma$. It is not hard to see that $\sat\Sigma NM$ is the kernel of $M\to \inv\Sigma(M/N)$. In particular, it is again an $R$-module. In case $N$ is cyclic, generated by a single element $n$, we may write $\sat\Sigma nM$ instead of $\sat\Sigma NM$.
%
%
%

We will use this construction in the following situation. Let $R$ be a complete Noetherian local ring of \ch\ $p>0$ and let $S\sub R$ be some Noether normalization of $R$, that is to say, a regular local subring $S\sub R$, over which $R$ is finite (as an $S$-module). Recall that by Cohen's structure theorem, Noether normalizations which have moreover the same residue field, are completely determined by a choice of system of parameters in $R$,  by having them act as the variables in a power series over the residue field. Note that $S\setminus\{0\}$ is a multiplicative set in $R$, which, for simplicity, we just denote by $S$.

\begin{proposition}\label{P:QSR}
If $R$ is a complete local domain of \ch\ $p>0$, if $q$ is any power of $p$, and if $S\sub R$ is some Noether normalization of $R$, then $\sat S{*1}{\frob{q*}R}$ is the $\pint qR$-submodule of $\frob{q*}R$ generated by $*1$, and hence $\sat S{*1}{\frob{q*}R}\iso \pint qR$, so that in  particular it is independent from the Noether normalization $S$.
\end{proposition}
\begin{proof}
An element in the $\pint qR$-submodule of $\frob*R$  generated by $*1$ is of the form $f{*}1$ for some $f\in \pint q R$. The latter means that $f^q=a\in R$, and so $f{*1}=*a$. Write $f=r/t$ with $r,t\in R$ and $t\neq 0$. Since $R$ is finite over $S$, there exists $t'\in R$ such that $s:=tt'$ is a non-zero element of $S$, and hence $sf=rt'$. Therefore, $s{*a}=sf{*1}=rt'{*}1$, showing that $*a\in \sat S{*1}{\frob{*q}R}$. The converse follows along the same lines: take $*a\in \sat S{*1}{\frob{q*}R}$, so that $s{*}a=r{*1}$ for some non-zero $s\in S$ and $a, r\in R$. This means that $*s^qa=*r^q$, whence $s^qa=r^q$, so that $a=(r/s)^q$, and hence $f:=r/s\in \pint qR$. It follows that $*a=*f^q=f{*}1$, as we needed to show. 
\end{proof}

\begin{remark}\label{R:QSR}
Even if $R$ is not a domain, it is easy to see that $\sat S{*1}{\frob{q*}R}$ is closed under multiplication, and hence is in fact an $R$-subalgebra of $\frob*R$. 
\end{remark} 

Suppose again that $R$ is a domain. Let us call, in general, regardless of \ch, an element $f\in K$ \emph{power-integral}, if $f^m\in K$, for all $m\gg0$. Clearly, if $R$ has positive \ch, then a power-integral element is F-integral. Note, however, that   power-integral elements are not closed under addition and so in general do not form a ring; instead, we have to take the $R$-algebra they generate, denoted $\powint R$. Clearly, $\powint R$ lies in the integral closure $\tilde R$ of $R$, and if $R$ has positive \ch, then 
\begin{equation}\label{eq:powintFint}
\powint R\sub\pint{}R. 
\end{equation} 
We expect that in the toric case, \eqref{eq:powintFint} is an equality (see \Examp{E:e3}). For instance, in the affine case,  we know that $\tilde R$ is again toric. Suppose $p>[\tilde R: R]$, and let  $f\in K$ be F-integral. Let $\Gamma(f)\sub \nat$ be the semi-group of exponents $m$ such that $f^m\in R$. Hence for some $n\gg0$, we have $p^n\in\Gamma(f)$.  Since $f\in\tilde R$, 
it satisfies a binomial integral equation   of the form $T^m-u\tuple y^\alpha$ with $u$ a  unit, we also get $m\in\Gamma(f)$. By assumption $m$ and $p^n$ are co-prime, so that $\Gamma(f)$ is co-finite, which means  that $f$ is power-integral. Note, however, that it is not even clear whether $\powint R=\powint{(\powint R)}$ (it will be, if  \eqref{eq:powintFint} is  an equality).

\section{Hochster rings in positive \ch}
In this section, we will provide a strategy for constructing small MCM's in \ch\ $p$. We     only will illustrate the method for the   families $\mathcal T_{d,n,e}(\kappa)$ from the introduction, but presumably, many other cases can be treated this way (we work out one such more general case in \Examp{E:genfam}). To start, we need a number-theoretic lemma:

\begin{lemma}\label{L:padicm}
Given $q,m,b$ with $q\equiv1\mod m$ and $b<q$,  let $b(\tfrac{q-1}m)=b^{(1)}q+b^{(0)}$ be written in $q$-adic expansion, so that $0\leq b^{(i)}<q$. Then $b^{(1)}+b^{(0)}<q$, and in particular, $b^{(1)}q^2+(b^{(1)}+b^{(0)})q+b^{(0)}$ is the $q$-adic expansion of $b(\tfrac{q^2-1}m)$.

In fact,  if we let $\epsilon$ be the `adjusted' remainder of $b$ modulo $m$, that is to say, $1\leq\epsilon\leq m$ and $b\equiv\epsilon\mod m$, then $\tr b0=\tfrac {q\epsilon-b}m$.
\end{lemma}
\begin{proof}
Write $q=sm+1$ and $b=tm+\epsilon$ with $1\leq \epsilon\leq m$. Since $b<q$, we get $t\leq s$. Hence $b^{(1)}q+b^{(0)}=b\tfrac{q-1}m=(tm+\epsilon)s=tsm+s\epsilon  =t(q-1)+r=tq+s\epsilon -t$. Since  $\epsilon\leq m$, we get $s\epsilon-t\leq s\epsilon\leq sm<q$.   Since $t\leq s$, we get $s\epsilon-t\geq 0$ and comparing $q$-adic expansions, we see that $ b^{(1)}=t$ and $b^{(0)}=s\epsilon-t$, and so their sum is equal to $s\epsilon<q$. 
The second assertion follows since $b\tfrac{q^2-1}m=(q+1)b\tfrac{q-1}m$.
\end{proof} 
\begin{remark}\label{R:padicm}
Given $q>1$, define the \emph{$q$-adic trace} $\trace qa$ of a number $a\in\nat$ as the sum of its $q$-adic digits. The above proof shows that if $q\equiv 1\mod m$ and $b<q$ has adjusted remainder $\epsilon$ modulo $m$, then
$$
\trace q{b\tfrac{q-1}m}= \epsilon\tfrac{q-1}m
$$
 Moreover, 
$$
\trace q{b\tfrac{q^2-1}m}=2\trace q{b\tfrac{q-1}m}.
$$
\end{remark}

\begin{theorem}[The case $m=2$]\label{T:quadHochstergen}
Let $S=\pow \kappa{\tuple y}$ be a power series in $d$ variables ${\tuple y}$ where $\kappa$ is    field     of \ch\ $p\neq 2$.  Let $\tuple u$ be an $n$-tuple of variables, let $a_i,b_{ij}\in S$ be units and assume $a_ia_j=b_{ij}^2$ for $i<j$. Let $\alpha_i$ be $d$-tuples of (positive) exponents, such that $\alpha_i+\alpha_j=2\beta_{ij}$, for   some $\beta_{ij}$ and all $i<j$. Let $I\sub\pow S{\tuple u}$ be the ideal generated by all $u_i^2-a_i\tuple y^{\alpha_i}$ and $u_iu_j-b_{ij}\tuple y^{\beta_{ij}}$ for $1\leq i<j\leq n$. Then  $R:=\pow S{\tuple u}/I$ admits a small MCM.
\end{theorem} 
\begin{proof}
Note that $R$ is a member of the family $\mathcal T_{d,n,2}(\kappa)$ from the introduction.
As an $S$-module, $R$ is minimally generated by $1$ and all $u_i$, but not freely, since we have the following syzygy relations
\begin{equation}\label{eq:uisyz}
a_i\tuple y^{\alpha_i}u_j=b_{ij}\tuple y^{\beta_{ij}}u_i
\end{equation} 
for all $i<j$. 
In particular, $R$ is not free as an $S$-module, whence is not \CM. Moreover, the $S$-submodule $N$ generated by all $u_i$   is a direct summand, and $R= S\oplus N$ as $S$-modules (where $S\sub R$ in the natural way). 

Let $q$ be a power of $p$ which is bigger than any entry in $2\beta_{ij}$. We will use the Frobenius with respect to $q$, but to not overload notation, we simply will assume that $p$ itself is sufficiently big and leave the details for small primes to the reader.  Since the base change $R':=R\tensor_S\frob*S$ is finite over $R$, it suffices to show that $R'$ admits a small MCM (note that as a ring $\frob*S\iso S$, so that $R'$ is also a member of $\mathcal T_{d,n,2}(\kappa)$). Since   the $a_i$ are $p$-th powers in $R'$, we may assume from the start that each $a_i$ has a $p$-th root in $S$. 
Our goal is to show that  $Q:=\sat S{*1}{\frob*R}$ is a (small) MCM, and to this end, we only need to show that $Q\iso S^2$ for then its depth as an $S$-module, whence as an $R$-module, would be $d$.   Note that $\frob*R$ is (non-freely) generated as an $S$-module by $\{*\tuple y^\delta,*u_i\tuple y^\delta\}$, where $\delta$ runs over all exponent vectors in $\{0,\dots,p-1\}^d$ and where $1\leq i\leq n$. As we will see (see, for instance, \eqref{eq:eiej} below), this set is not even minimally generating.

To verify the claim, for any $\delta\in\{0,\dots,p-1\}^d$, let $\tr\delta1p+\tr\delta0$ be the $p$-adic expansion of $\delta\cdot(\tfrac{p-1}2)$.  Hence
$$
u_i{*}1=*u_i^p=*u_ia_i^{\tfrac{p-1}2}\tuple y^{\tr{\alpha_i}1p+\tr{\alpha_i}0}=a_i^{\tfrac{p-1}{2p}}\tuple y^{\tr{\alpha_i}1}{*}u_i\tuple y^{\tr{\alpha_i}0}
$$
Put $e_0:=*1$ and $e_i:={*}u_i\tuple y^{\tr{\alpha_i}0}$, so that, by definition of $S$-saturation, the $e_i$, for $i=\range 0n$, generate $Q$.
We will show that up to a unit, the $e_i$, for $i=\range 1n$ are all the same elements, showing that $Q$ is in  fact generated as an $S$-module by two elements, $e_0$ and $e_1$, which are easily seen to be linearly independent over $S$, as we wanted to show. To verify the latter claim, notice that our assumptions imply that all $\alpha_i$ have the same adjusted remainder $\epsilon$ modulo two (recall that this means that $\epsilon$ has entries either one or two, depending on the parity of the corresponding entry in the tuples $\alpha_i$). 
By \Lem{L:padicm}, we have
$$
\tr{\alpha_i}0=\frac{p\epsilon-\alpha_i}2.
$$
Since we chose $p$ large enough, the quantity 
$$
\frac{p\epsilon-\alpha_i}2-\beta_{ij}=\frac{p\epsilon-\alpha_j}2-\alpha_i
$$
 is positive, so that from \eqref{eq:uisyz}, we get
\begin{equation}\label{eq:eiej}
b_{ij}e_i=b_{ij}u_i\tuple y^{\beta_{ij}}\tuple y ^{\tfrac{p\epsilon-\alpha_j}2-\alpha_i}=a_iu_j\tuple y ^{\tfrac{p\epsilon-\alpha_j}2}=a_ie_j
\end{equation} 
proving the claim.
\end{proof} 

\begin{theorem}[The case $m$ odd]\label{T:binommHoch}
Let $S=\pow \kappa{\tuple y}$ be a power series in $d$ variables ${\tuple y}$ over a field $\kappa$ of   \ch\ $p$ not dividing  $m\in\nat$. Let $a, b,c\in S$ be units and assume $ab=c^m$. Let $\alpha, \beta$ be $d$-tuples of (positive) exponents, such that $\alpha+\beta=m\gamma$, for   some $\gamma$. Let $I\sub\pow S{u,v}$ be the ideal generated by   $u^m-a\tuple y^{\alpha}$, $v^m-b\tuple y^{\beta}$ and $uv-c\tuple y^{\gamma}$. Then  $R:=\pow S{u,v}/I$  admits a small MCM.
\end{theorem} 
\begin{proof}
After taking some finite  extension of $\kappa$, we may assume that the residues of $a,b$ in $\kappa$ are $m$-th powers, and hence, by Hensel's Lemma, there exist units $a_0,b_0\in S$ such that 
$a=a_0^m$ and $b=b_0^m$. After the change of variables $u\mapsto u/a_0$ and $v\mapsto v/b_0$, we may therefore assume that they are equal to one. Let $N$ be the maximum of the entries in $\alpha$ and $\beta$ and choose    a power $q$ of $p$ such that $q\equiv 1\mod m$ and $q>mN$. We will use the Frobenius with respect to $q$. 
As in the previous proof,  we show that $Q:=\sat S{*1}{\frob*R}$ is a (small) MCM over $R$, by proving that as an $S$-module,  $Q\iso S^m$. 
As generators of $\frob*R$  we take again the standard generating set  $\mathbf E:=\{u^i\tuple y^\delta,v^i\tuple y^\delta\}$ with $0\leq i<m$ and  $\delta\in\{0,\dots,q-1\}^d$.

 Put $e_{0,0}:=*1$. Since the defining equations are binomial, there exists, for each pair $(i,j)\in\nat^2$ a generator $e_{ij}\in \mathbf E$ and some  $s_{ij}\in S$ such that $u^iv^je_{0,0}=s_{ij}e_{ij}$, and these $e_{ij}$ then generate $Q$ over $S$. We will show that already $e_{0,0},e_{1,0},\dots,e_{m-1,0}$ generate $Q$ as an $S$-module. Since $1,u,\dots,u^{m-1}$ are linearly dependent over $S$, this shows the claim that $Q\iso S^m$. From the defining equations, it follows that $e_{m,0}=e_{0,m}=e_{1,1}$, etc.  Hence the $e_{i,0}$ and $e_{0,i}$ for $i<m$ already generate $Q$. So remains to show that 
 \begin{equation}\label{eq:eimi}
e_{0,i}=e_{m-i,0},
\end{equation} 
for all $1\leq i\leq m-1$. 

To this end, let $r:=\tfrac{q-1}m$ and define for any $a<q$ and any $0<i<m$, the remainder of $ria$ modulo $q$ by $\tr {a_i}0$, and the adjusted remainder of $ia$ modulo $m$ by $\tilde a_i$ (and a similar notation for tuples).  It follows from \Lem{L:padicm} that
\begin{equation}\label{eq:tria0}
\tr{a_i}0 = \tfrac{q\tilde a_i-ia}m
\end{equation} 
Moreover, if $a+b\equiv0\mod m$, then $ia\equiv(m-i)b$, showing that 
\begin{equation}\label{eq:aibmi}
\tilde a_i=\tilde b_{m-i}.
\end{equation} 
To determine  $e_{i,0}$, we must calculate 
$$
u^i{*}1=*u^{qi}=*u^i\tuple y^{ri\alpha}=s_i{*}u^i\tuple y^{\tr{\alpha_i}0}
$$
for some $s_i\in S$, showing that $e_{i,0}={*}u^i\tuple y^{\tr{\alpha_i}0}$. Likewise, we have $e_{0,i}={*}v^i\tuple y^{\tr{\beta_i}0}$. From the identity $u^i(uv)^{m-i}=v^{m-i}u^m$, we get a relation
\begin{equation}\label{eq:syzimi}
u^i\tuple y^{(m-i)\gamma}=v^{m-i}\tuple y^{\alpha}
\end{equation} 
Fix $i$ and let $\delta:= \tr{\beta_{m-i}}0-\alpha$. Since we choose $q$ big enough, one verifies, using \eqref{eq:tria0} for $\beta$, that $\delta\geq0$. Furthermore, using \eqref{eq:aibmi}, we get
$$
\delta=\tfrac{p\tilde \beta_{m-i}-(m-i)\beta}m -\alpha=\tfrac{p\tilde \alpha_i-i\alpha}m-(m-i)\gamma=\tr{\alpha_i}0-(m-i)\gamma.
$$
Therefore, we can multiply  both sides of syzygy~\eqref{eq:syzimi} with $\tuple y^\delta$, yielding
$$
u^i\tuple y^{\tr{\alpha_i}0}=v^{m-i}\tuple y^{\tr{\beta_{m-i}}0}
$$
whence proving \eqref{eq:eimi}.
\end{proof} 

\begin{remark}\label{R:binommHoch}
The condition that $p$ and $m$ be co-prime (=tameness in the terminology of \S\ref{s:bitor} below) is presumably not necessary (see \Examp{E:e3}), and one should be able to treat it in a similar  way. 
\end{remark} 

\begin{example}\label{E:e3}
Let me work out just one simple example in more detail. We take $m=3$, $d=3$, and look at the (purely) toric ring $R$ with defining equations $u^3=xy^2z^3$, $v^3=x^5yz^6$, and $uv=x^2yz^3$ over a field $\kappa$ of \ch\ $p$. Note that the highest degree is $N=6$. We have the following two syzygy relations 
\begin{equation}\label{eq:syze3}
vxy^2z^3=u^2x^2yz^3\qquad\text{and}\qquad ux^5yz^6=v^2x^2yz^3
\end{equation}  
It follows from \eqref{eq:syze3} that $\pr:=(vy-u^2x,v^2-ux^3z^3)R$  is a minimal prime ideal of $R$. Put $\bar R:=R/\pr$, so that $\bar R$ is a three-dimensional domain.
To calculate its integral closure, observe  that the morphism $\varphi\colon\bar R\to \pow\kappa {a,b,c}$ given by the parametrization 
$$\left\{
\begin{aligned}
x&= a^3\\
y&= b^3\\
z&= c\\
u&=ab^2c\\
v&=a^5bc^2
\end{aligned}
\right.
$$
is well-defined  on $\bar R$. Since source and image of $\varphi$ are both  three-dimensional domains, it must also be injective. Moreover,  $\tfrac uz$ and $\tfrac v{xz^2}$ are integral over $\bar R$, and $\varphi$  can be extended to the $\bar R$-algebra $ R'$ they generate by letting  $\tfrac uz\mapsto ab^2$ and $\tfrac v{xz^2}\mapsto a^2b$.  As the image of the latter morphism is the subring $\pow \kappa{a^3,a^2b,ab^2,b^3,c}$, we showed that $\op{Spec}( R')$ is non-singular, equal to  the product of the normal scroll of degree three with the affine line, and hence in particular, $R'$ is normal, whence the normalization of $\bar R$. As $R'$ is \CM, it is a small MCM for $R$ (this is true in general for complete purely toric rings by the theory of affine toric rings). We will, however, show that there is a smaller MCM given by the theorem. We consider various \ch{s}.

If $p=7$, then according to the proof, we should take $q=49$ to make it bigger than $3N=18$, but we shall see that already $q=7$ works. 
One calculates that $u{*}1=*ux^2y^4z^6$ and $u^2{*}1=yz{*}u^2x^4yz^5$, so that $Q$ is the submodule of $\frob*R$ generated by $e_0:=*1$, $e_1:=*ux^2y^4z^6$ and $e_2:=*u^2x^4yz^5$. Indeed,  $v{*}1=ve_0=xze_2$ and $v^2{*}1=v^2e_0=x^3z^3e_1$ together with $ue_0=e_1$ and $u^2e_0=yze_2$ are the relations among these generators. 
In particular, as an $S$-module,   $Q$ is generated by the three elements $e_0$, $e_1$, and $e_2$, and one easily checks that they are $S$-linearly independent, so that $Q\iso S^3$ is indeed an MCM for $R$.

   One verifies that $\pr$ annihilates $Q$, so that the latter is even  an MCM of $\bar R$ contained in $\frob*\bar R$, given by the same generators and relations over $\bar R$. In particular, $Q=\sat S{*1}{\frob*{\bar R}}\iso \pint 7{\bar R}$ by \Prop{P:QSR}.  The integral elements $\tfrac uz$ and $\tfrac v{xz^2}$ are not F-integral, showing that $R'$ is bigger than $Q$. On the other hand, the fraction   $\tfrac v{xz}=\tfrac{u^2}{yz}$ is $7$-integral and can be seen to generate $\pint 7{\bar R}=\pint{}{\bar R}$. Note that 
$\tfrac{u^2}{yz}$ is actually power-integral over $R$, whence  F-integral in any \ch, by \eqref{eq:powintFint}. In particular, 
$$
Q\iso \pint{}{\bar R}=\powint{\bar R}=\pol{\bar R}{\tfrac{u^2}{yz}}.
$$

Next take $p=11$. According to our proof, we should work with $q=121$ as this is equivalent to one modulo $3$, but as we shall see, we can again use the Frobenius with respect to $11$. Namely, let $e_0:=*1$, $e_1:=*ux^7y^3z^{10}$ and $e_2:=*u^2x^3y^6z^9$, then we have the following relations among these
\begin{equation}\label{eq:pintsyz}
ue_0=e_2\quad u^2e_0=yze_1\quad ve_0=xze_1\quad v^2e_0=x^3z^3e_2
\end{equation} 
showing that $e_0$ and $e_2$ generate  $Q$ (freely) as an $S$-module.  

For $p=13$, we cannot take $q=13$, since it will not give the necessary relations (explicitly $u{*_{13}}1$ and $v^2{*_{13}}1$ do not belong to the same cyclic $S$-module, whereas according to \eqref{eq:eimi} they should), and so we must take  $q=169$. If we define $e_1:=*ux^{56}y^{112}z^{168}$ and $e_2:=*u^2x^{112}y^{55}z^{167}$, then we have again the same relations~\eqref{eq:pintsyz}. Although $\tfrac{u^2}{yz}$   generates the F-normalization of $\bar R$, it cannot be embedded as a submodule of $\frob{13*}\bar R$, but only inside $\frob{169*}\bar R$. 
%
%

Although we only gave the proof for $p$ relative prime to $m$ (but see \Rem{R:binommHoch}), we can as easily check the case $p=3$, and we may already take $q=3$. In that case, choosing as generators for $Q$ the elements $e_0:=*1$, $e_1:=*xy^2$ and $e_2:=*x^2y$, we get relations $ue_0=ze_1$, $u^2e_0=yz^2e_2$, $ve_0=xz^2e_2$ and $ve_0=x^3z^4e_1$, so that $Q$ is in fact a free summand of $\frob*S$. However, something else is different: $\pr$ does no longer annihilate $Q$ and $\pint 3{\bar R}=R'$, since $\tfrac uz$ and $\tfrac v{xz^2}$ are $3$-integral. 
\end{example} 

\subsection{Bi-partite toric rings}\label{s:bitor}
The families $\mathcal T_{d,n,m}$ are special cases of the following more general construction of toric local rings: let $S$, as above, be an unramified $d$-dimensional regular local ring, for which we fix a regular system of parameters $\tuple y=\rij yd$, let $\Gamma\sub\nat^n$ be a (finitely generated) full semi-group (meaning that its divisible (group) hull is $\mathbb Q^n$), let $\varphi \colon \Gamma\to \nat^d$ be a \homo\ of semi-groups, and let $\chi$ be an \emph{$S$-character} on $\Gamma$, by which we mean a \homo\ from $\Gamma$ to the multiplicative  group $S^*$ of  units in $S$. Since $\Gamma$ is full, taking divisible hulls leads to    a linear transformation $\mathbb Q^n\to\mathbb Q^d$ extending  $\varphi$; let $\mat A_\varphi$ be the $n\times d$-matrix defining this transformation (with respect to the standard bases). This matrix is (\emph{positive})  \emph{integral} on $\Gamma$, meaning that $\gamma\mat A_\varphi\in\nat^d$, for all $\gamma\in\Gamma$, and conversely, any matrix over $\mathbb Q$ which is integral on $\Gamma$ induces a morphism of semi-groups as above. 

To this data we associate a  toric local ring $\mathcal R(S,\Gamma,\varphi,\chi)$ given as the quotient of $\pow S{\tuple u}$, where $\tuple u=\rij un$, modulo the toric ideal generated by all $\tuple u^\gamma-\chi(\gamma)\tuple y^{\varphi(\gamma)}$, for $\gamma\in\Gamma$.  For instance, \Examp{E:e3} is obtained this way by taking the semi-group  generated by $(3,0)$, $(1,1)$, and $(0,3)$, letting the character $\chi$ to be trivial, and letting $\varphi$ to be defined by the matrix
$$
\mat A_\varphi:=\left(\begin{array}{ccc} 1/3 &  2/3 & 1 \\ 5/3 &   1/3 & 2\end{array}\right).
$$   These toric rings have the property that their defining binomial equations are given by   partitioning the tuple of variables and equating a monomial in the first tuple to a  unit times a monomial in the second tuple, whence the name \emph{bi-partite}. In fact, they give a   `toric' Noether normalization:

\begin{lemma}\label{L:torNN}
With $S,\Gamma,\varphi,\chi$ as above, the natural map $S\to \mathcal R(S,\Gamma,\varphi,\chi)$ is injective and finite, i.e., a Noether normalization. Moreover, if the sum of all entries in $\mat A_\varphi$ is bigger than one (the general case), then the maximal ideal $\mathfrak n$ of $S$ is a reduction of the maximal ideal of $R$.
\end{lemma}
\begin{proof}
Let $d:=\dim S$ and $R:=\mathcal R(S,\Gamma,\varphi,\chi)$. By the fullness assumption, some positive multiple  of each basis vector $\epsilon_i\in\nat^n$ lies in $\Gamma$. Let $\gamma_i:=a_i\epsilon_i$ be the smallest such multiple. The binomial equation corresponding to $\gamma_i$ is then $u_i^{a_i}-\chi(\gamma_i)\tuple y^{\gamma_i}$, showing that $u_i$ is integral over $S$.   This already shows that the map $S\to R$  is finite, and hence $\dim R\leq d$. To prove injectivity, since $S$ is a domain, we only need to show that $\dim R\geq d$. To this end, we may make a flat base change, and so assume that $S$ is complete with algebraically closed residue field. By Cohen's structure theorem, it is therefore  a power series over $V$, where $V$ is either a field or a complete $p$-ring. In the second, mixed \ch\ case, we can adjoin one extra variable replacing the role of the regular parameter $p$, and hence assume from the start that $\tuple y$ are variables in $S$ (see the last part in the proof of \Thm{T:mixtorMCM} for more details).

Let us prove the inequality $d\leq \dim R$ first in the purely toric case, i.e., when $\chi=1$ is trivial.
Let $a$ be the common denominator of the $a_i$, and let $\tfrac{\alpha_1}a,\dots,\tfrac{\alpha_d}a$ be the rows of the matrix $\mat A_\varphi$ with $\alpha_i\in\nat^d$. Consider the $V$-algebra morphism $s\colon \pow S{\tuple u}\to S$ given by $y_i\mapsto y_i^a$ 
and $u_i\mapsto \tuple y^{\alpha_i}$. For $\gamma=\rij gd\in \Gamma$, the image of $\tuple u^{\gamma}$ under $s$ is $\tuple y^{\beta}$ with $\beta=g_1\alpha_1+\dots+g_d\alpha_d=a\varphi(\gamma)$, whence is equal to the image of $\tuple y^{\varphi(\gamma)}$ under $s$. Hence $s$ factors through $R$, and since  $s$ is   finite, we get $\dim R\geq d$, as required. 

So assume now that we have  a non-trivial  character $\chi$. Let $l_i\in V^*$ be the constant term of $\chi(\gamma_i)\in S^*$, for $i=\range 1n$. Since we assumed the residue field to be algebraically closed, we can find $k_i\in V$ such that $k_i^{a_i}=l_i$ (in the mixed \ch\ case, we may  also have to take a finite ramification of $V$, but this does not change the dimension). Let $\tuple z$ be a new $n$-tuple of variables and let $S'$ be the localization of $\pol S{\tuple z}$ at the prime ideal generated by the $y_i$ and $z_i-k_i$, so that $S'$ is regular of dimension $d+n$. Let $\varphi'$ be the linear transformation with $n\times(n+d)$-matrix $(\mat A_\varphi|\mat I_n)$, where $\mat I_n$ is the $n\times n$-identity matrix, and let $R':=\mathcal R(S,\Gamma,\varphi',1)$ (note that $\varphi'$ is indeed integral on $\Gamma$). By the purely toric case, $R'$ has dimension $d+n$. By construction, each $z_i^{a_i}-\chi(\gamma_i)$ is a non-unit, and hence the ring $\bar R':=R'/(z_1^{a_1}-\chi(\gamma_1),\dots,z_n^{a_n}-\chi(\gamma_n))R'$ has dimension at least $d$ by Krull's principal ideal theorem. So remains to show that $R=\bar R'$. For $\gamma=\rij gn\in \Gamma$, the corresponding defining equation for $R$ is $\tuple u^\gamma-\chi(\gamma)\tuple y^{\varphi(\gamma)}$, whereas for $R'$, it  is $\tuple u^\gamma-\tuple y^{\varphi(\gamma)}\tuple z^\gamma$, and we want  to show that they are the same in $\bar R'$. Indeed, in $\bar R'$, we have
$$
\begin{aligned}
\tuple z^\gamma&=z_1^{g_1}\cdots z_n^{g_n}\\
&=(z_1^{a_1})^{\frac{g_1}{a_1}}\cdots (z_n^{a_n})^{\frac{g_n}{a_n}}\\
&=\chi(\gamma_1)^{\frac{g_1}{a_1}}\cdots \chi(\gamma_n)^{\frac{g_n}{a_n}}\\
&=\chi({\frac{g_1}{a_1}}\gamma_1+\dots+{\frac{g_n}{a_n}}\gamma_n)=\chi(\gamma).
\end{aligned}
$$

As for the last assertion, let   $\maxim$ be the   maximal ideal of $R$. The condition on $\mat A_\varphi$ implies that $\norm{\varphi(\gamma)}\geq \norm\gamma$, where we write $\norm\cdot$ to mean the sum of all entries of a tuple. It follows that a monomial  in $\tuple u$ of degree $m$ lies in $\mathfrak n^{m+1}R$, for all $m$ bigger than some $m_0$. Hence $\maxim^m\mathfrak n^{m_0}R=\maxim^{m+m_0}$. 
\end{proof} 
\begin{remark}\label{R:nonintdom}
Without proof, we state that the kernel of $s$ is the unique $d$-dimensional prime ideal $\pr$ of $R$, provided  each $a_i>1$ (lest we would get  binomial equations that are trivially reducible). In particular,  the localization $\pol R{\tfrac 1{\tuple y}}$ is a domain and $\pr$ is the kernel of the localization map $R\to \pol R{\tfrac 1{\tuple y}}$. 
\end{remark} 
\begin{remark}\label{R:bipardeg}
By the lemma,  the regular system of parameters $\tuple y$ of $S$  is  a system of parameters on $R$. Therefore, the parameter degree of $R$ is at most the length $\ell(R/\mathfrak nR)$ of the closed fiber,    and I postulate that they are actually equal (this would follow from the more general conjecture that in an arbitrary local ring $(R,\maxim)$, if $\mathfrak q$ is a parameter ideal which is also a reduction of $\maxim$, then $\op{pardeg}(R)=\ell(R/\mathfrak q)$).
\end{remark}

Let $p$ be the \ch\ of the residue field of $S$. We  call $\varphi$ \emph{tame} if $\mat A_\varphi$ has linearly independent rows (whence $n\leq d$) and  no entry in $\mat A_\varphi$ has denominator divisible by $p$ (this holds trivially if $p=0$). In particular,  $\mat A_\varphi$ is then defined over the residue field. The rings in \Thm{T:binommHoch} are tame, but it is not clear yet whether we need this assumption (see \Rem{R:binommHoch} and  \Examp{E:genfam} below). Nonetheless,  we will show in \Thm{T:torNNmix}  below that the mixed \ch\ `big' analogue of the following conjecture holds under the additional   tameness condition.

\begin{conjecture}\label{C:torNNMCM}
Let $S$ be a regular local ring of \ch\ $p>0$, 
let $\Gamma\sub \nat^n$ be a full semi-group, and let $\varphi\colon\Gamma\to \nat^d$ and $\chi\colon\Gamma\to S^*$ be respectively a \homo\   and an $S$-character. Then the bi-partite toric ring $R:=\mathcal R(S,\Gamma,\varphi,\chi)$ defined by this data admits a small MCM, to wit, $\sat S{*1}{\frob{q*}R}$, for some sufficiently high power $q$ of $p$. 

Moreover, if we are in the situation described in \Rem{R:nonintdom}, so that $R$ has a unique $d$-dimensional   prime  $\pr$, and if $\varphi$ is tame,\footnote{\Examp{E:e3} shows the necessity of this latter assumption.} then   $\pr$ annihilates $\sat S{*1}{\frob{q*}R}$ and 
\begin{equation}\label{eq:powintbip}
\sat S{*1}{\frob{q*}R}\iso \pint{}{R/\pr}=\powint{(R/\pr)}.
\end{equation} 
\end{conjecture}

Can we prove \Conj{C:FRMCM} from this conjecture, using \Lem{L:qdefHoc}, by showing that any toric local ring is a \qdef\ of a bi-partite one?

\begin{example}\label{E:genfam}
Let me just work out one bi-partite example for which \Conj{C:torNNMCM} holds. Let $\Gamma$ be the (full) semi-group of $\nat^2$ generated by $(2,0)$, $(1,3)$, and $(0,6)$, and let $\varphi$ be the \homo\ into $\nat^3$ given by the matrix
$$
\mat A_\varphi = \left(\begin{array}{ccc}1/2 & 1 & 2 \\5/6 & 1 & 1/3\end{array}\right).
$$
It is positive integral on $\Gamma$, as it sends $(2,0)$, $(1,3)$, and $(0,6)$ respectively to $(1,2,4)$, $(3,4,3)$ and $(5,6,2)$. For simplicity, let $\chi$ be the trivial character. 
In terms of equations, using $(u,v)$ as the variables for $\Gamma$ and $(x,y,z)$ for  $\nat^3$, we get the three (bi-partite) binomial equations 
$$
u^2=xy^2z^4,\quad uv^3=x^3y^4z^3,\quad v^6=x^5y^6z^2.
$$  
 When $p=11$,  we can   take the Frobenius with $q=11$. Let $Q:=\sat S{*1}{\frob{11*}R}$. Note that $R/(x,y,z)R$ is an Artinian ring of length $9$ (and hence, by \Rem{R:bipardeg}, we expect that $\op{pardeg}(R)=9$), with (monomial) basis 
\begin{equation}\label{eq:monbas}
 \{1,u,v,uv,v^2,uv^2,v^3,v^4,v^5\},
\end{equation} 
  and it suffices to know the action of these nine elements on $*1$ to calculate $Q$. Define
 $$
\xymatrix@R=10pt{
 e_0:=*1 &e_1:=*vxy^{10}z^7 &e_2:=*v^2x^2y^9z^3 \\
e_3:=*v^3x^3y^8z^{10} &e_4:=*v^4x^4y^9z^6 &e_5:=*v^5x^5y^6z^2
 }
 $$
 then the action of the  basis~\eqref{eq:monbas} on $e_0$ gives respectively 
 \begin{equation}\label{eq:actionuv}
\xymatrix@R=10pt@C=7pt{
 1\ar@{~>}[d]&u\ar@{~>}[d]&v\ar@{~>}[d]&uv\ar@{~>}[d]&v^2\ar@{~>}[d]&uv^2\ar@{~>}[d]&v^3\ar@{~>}[d]&v^4\ar@{~>}[d]&v^5\ar@{~>}[d]\\
 e_0&ze_3&e_5&xyz^2e_2&xye_4&x^2y^2z^2e_1&x^2y^2e_3&x^3y^3ze_2&x^4y^4ze_1,
 }
\end{equation} 
  showing that $Q$ is   generated by $e_0,\dots,e_5$ as an $S$-module. Since the $v^i$, for $i=\range 05$ are linearly independent, $Q\iso S^6$, whence an MCM for $R$. It follows from the last assertion in \Lem{L:torNN} and \cite[Theorem 14.13]{Mats}, that $\op{mult}(Q)=\op{mult}(\tuple yR, Q)\leq Q/\tuple yQ\iso \kappa^6$, showing that $Q$ is \vs. 
  
   Let $\bar R$ be the residue ring modulo a three dimensional  minimal prime. It follows from the syzygy relation $ux^3y^4z^3=v^3xy^2z^4$ that it must contain $ux^2y^2-v^3z$, and without proof, we state that this latter binomial generates the minimal prime and---for reasons I do not understand yet---again kills $Q$, making $Q$ into an $\bar R$-algebra, namely $\pint{11}{\bar R}$. To find  $11$-integral elements over $\bar R$, we take a basis element from the first row in \eqref{eq:actionuv} and divide it  by the corresponding coefficient in the second row. Modulo the relation $ux^2y^2=v^3z$, we find four $11$-integral elements, each of which is in fact power-integral. However, this does not give all F-integral elements: for instance, $\tfrac{v^2}{xy^2}$ is also power-integral (the fifth relation in \eqref{eq:actionuv} only gives the element $\tfrac{v^2}{xy}$). On the other hand, $\tfrac {v^3}{x^2y^3z}$ is integral, but not F-integral (for $p\neq 2$), showing that $Q$ is not the normalization of $\bar R$.
%
   
When $p=q=13$, we get the same relations~\eqref{eq:actionuv}, provided we now take as (free) generators for $Q$, the elements
$$
\xymatrix@R=10pt{
 e_0:=*1 &e_1:=*v^5x^{11}y^{8}z^7 &e_2:=*v^4xy^9z^3 \\
e_3:=*v^3x^4y^{10}z^{12} &e_4:=*v^2x^7y^{11}z^8 &e_5:=*vx^{10}y^{12}z^4
 }
 $$
 and again $\pr Q=0$ and $Q\iso \powint{\bar R}$.
 
 Tameness is not an obstruction in this case, for when $p=3$, we can take $q=9$, and when we let$$
\xymatrix@R=10pt{
 e_0:=*1 &e_1:=*v^3x^{8}y^{6}z^5 &e_2:=*x^3z^3 \\
e_3:=*v^3x^2y^{6}z^{8} &e_4:=*x^6z^6 &e_5:=*v^3x^{5}y^{6}z^2
 }
 $$
 then they too freely generate $Q$ over $S$. Unlike the previous case, however, half of the generators belong already to the submodule $\frob*S$; moreover, we have slightly different relations than   \eqref{eq:actionuv} (differences are marked in boldface):
 $$
 \xymatrix@R=10pt@C=7pt{
 1\ar@{~>}[d]&u\ar@{~>}[d]&v\ar@{~>}[d]&uv\ar@{~>}[d]&v^2\ar@{~>}[d]&uv^2\ar@{~>}[d]&v^3\ar@{~>}[d]&v^4\ar@{~>}[d]&v^5\ar@{~>}[d]\\
 e_0&ze_3&e_5&xy^{\mathbf 2}z^2e_2&xy^{\mathbf 2}e_4&\mathbf{x}y^2z^2e_1&x^2y^2e_3&x^3y^{\mathbf 4}ze_2&x^{\mathbf 3}y^4ze_1.
 }
 $$
  These relations, for instance, demonstrate the F-integrality of $\tfrac{v^2}{xy^2}$, which we did not see for $p=11,13$. It is therefore not clear yet how to calculate $\pint {}R$ in general. Nonetheless, unlike what happened in the non-tame case in \Examp{E:e3}, we still have $\pr Q=0$ here. 
   \end{example} 
%
%

\section{Transfer to equal \ch\ zero}
Recall that we called $R$  \emph{strongly Hochster}, if some scalar extension of $R$ admits a \vs\  MCM, that is to say, a small MCM whose multiplicity is at most the equi-parameter degree of $R$. In equal \ch, this is the same as the    \emph{parameter degree} of $R$, that is to say,  the least possible length $\ell(R/I)$ where $I$ runs over all parameter ideals. 
Clearly, if $R$ is \CM, whence parameter degree and multiplicity agree, a \vs\  MCM exists, to wit, $R$ itself, and so any local \CM\   ring is strongly Hochster.

\subsection*{Proof of \Thm{T:Hochequi}.}
Let $R$ be a $d$-dimensional Noetherian local ring of equal \ch\ zero with residue field $\kappa$ and parameter degree $\rho$. Passing to  the completion and then modding out a minimal  prime ideal of dimension $d$, we may moreover assume that $R$ is a complete domain. By \cite{SchAsc} or \cite[Theorem 7.1.6]{SchUlBook}, we can find a Lefschetz's hull $L(R)$ of $R$, meaning that $R\to L(R)$ is a scalar extension and $L(R)$ is the ultraproduct of $d$-dimensional complete Noetherian local domains $R_w$ of  positive \ch\ $p_w$ (and necessarily the $p_w$ are unbounded; note that $L(R)$, however, will not be Noetherian). Moreover, we may choose the $R_w$ so that they have  parameter degree $\rho$ as well. By assumption, each $R_w$ admits a scalar extension over which a small MCM module of multiplicity at most $\rho$ exists. Replacing each $R_w$ with this extension and taking their ultraproduct yields a possibly bigger scalar extension of $R$, and so we may, from the start, assume that each $R_w$   already admits a small MCM $Q_w$ of multiplicity at most $\rho$. By \cite[Prop.~3.5]{SchABCM}, we can find a $d$-dimensional regular subring $(S_w,\mathfrak n_w)$ of $  R_w$ such that $R/\mathfrak n_w R_w$ has length $\rho$. In particular, $Q_w$, having depth $d$,  is free as an $S_w$-module, say $Q_w\iso S_w^{N_w}$. In particular, $Q_w$ has rank at most $N_w$ as an $R_w$-module. By \cite[Theorem 14.8]{Mats}, the multiplicity is an upper bound for $N_w$, showing that $N_w\leq \rho$, and so, for almost all $w$, they are the same, say equal to $N$. Therefore, the ultraproduct $\ul Q$ of the $Q_w$ will be a finitely generated $L(R)$-module and by \los, free as an $L(S):=\ul S$-module. Let $\cp S$, $\cp R$,  and $\cp Q$ be the respective cataproducts of $S_w$, $R_w$ and $Q_w$, obtained by modding out the ideal of infinitesimals (=intersection of the powers $\ul\maxim^n$ of the maximal ideal). By \cite[3.2.12]{SchUlBook}, the \homo{s} $S\to \cp S$ and $R\to \cp R$ are scalar  extensions.\footnote{This might be also a good occasion to rectify some reference omissions  from \cite{SchUlBook,SchFinEmb}: I was not aware at that time that the construction and properties of cataproducts, including the result just quoted,  already appeared  under the name of \emph{separated ultraproducts} in \cite{MatPopUlBCM,MatPopUlMod,YoBCM}.} Since $\ul S\sub \ul R$ is   finite  (of   degree $\rho$),   so is $\cp S\to \cp R$. Moreover,   $\cp Q$ is a free $\cp S$-module, whence an MCM over $\cp R$.
\qed
\begin{remark}\label{R:Hochequi}
The proof shows more generally that given a class  $\mathcal T$ of complete local rings closed under scalar extensions and cataproducts,   if each member  of $\mathcal T$ of positive \ch\ is strongly Hochster, then so is each member  of equal \ch\ zero. In particular, this applies to the family   $\mathcal T_{d,n,e,N}$ consisting of any complete toric ring in some  $\mathcal T_{d,n,e}(\kappa)$, with $\kappa$ a field,  whose defining binomial equations have degree at most $N$, thus proving the equal \ch\ zero part of \Thm{T:dne}. 
\end{remark} 

\begin{remark}\label{R:pardeg}
For the ultraproduct argument to go through, we could weaken the  \vs\ condition to the following: given a function $f\colon \nat\to \nat$, let us say that $R$   is \emph{$f$-Hochster}, if it admits a small MCM of multiplicity at most $f(\rho)$, where $\rho:=\op{pardeg}(R)$ is the parameter degree of $R$. Let $\mathcal T$ be a class of complete local rings as in \Rem{R:Hochequi}. If for some function $f$,  each member of $\mathcal T$ of positive \ch\     is $f$-Hochster, then so is every member of equal \ch\ zero.   
\end{remark}

\begin{remark}\label{R:powinteq}
\Examp{E:e3} and the discussion after \eqref{eq:powintFint} suggest that we can postulate a more concrete MCM in equal \ch\ zero as well. Namely, for a non-degenerate $R\in\mathcal T_{d,n,e}(\kappa)$ of \ch\ $p$,   when $p>0$,   there should be a unique $d$-dimensional (toric) prime ideal $\pr$ (whose defining equations do not depend on $p$) and $\pint{}{R/\pr}=\powint{(R/\pr)}$ is a small MCM over $R$. Therefore, the same should be true in the cataproduct when $p=0$. In other words, in equal \ch\ zero, $R$ admits a small MCM realized as the ring of power-integral elements over (the unmixed part of) $R$. As we saw, this would   even in the affine case give a `smaller' MCM, and so we postulate the following version of Hochster's classical theorem (in any \ch):
\end{remark} 

\begin{conjecture}\label{C:afftorpowint}
If $A$ is the coordinate ring of an affine toric variety $X\sub\mathbb A_\kappa^n$, then $\powint A$ is \CM. 
\end{conjecture}

\section{Big MCM's in mixed \ch}\label{s:mix}

In this section, we work over a fixed  complete $p$-ring $V$ (meaning that $V$ is a complete \DVR\ with uniformizing parameter $p$), with perfect residue field $\kappa$. We will mainly be concerned with complete $V$-algebras $R$ that are torsion-free (whence flat) over $V$. For any such $V$-algebra, we will write $\bar R$ for $R/pR$, which therefore is a complete $\kappa$-algebra of dimension one less than $R$. We will write $\wi \cdot$ for the ($p$-typical) Witt functor (the reader may consult \cite{HazForm} or the notes \cite{RabWitt}). In particular,   $\wi \kappa=V$, and more generally, given a $\kappa$-algebra $A$, we get a $V$-algebra $\wi A$ and an epimorphism $\wi A\onto A$ (which is just reduction modulo $p$ if $A$ is moreover perfect). The latter has a categorically defined multiplicative section $A\to \wi A$, which we will simply denote by $\tau$ and call its \emph{Teichmuller character}.

Fix a tuple of variables $\tuple y=\rij yn$ and put $S_n:=\pow V{\tuple y}$. Hence $\bar S_n:=\pow\kappa {\tuple y}$. 
 We are interested in properties of the ring $B_n:=\wi{\bar S_n^+}$,  here $\bar S_n^+$ is the absolute integral closure of $\bar S_n$ (=the integral closure of $\bar S$ inside the algebraic closure $\kappa(({\tuple y}))^{\text{alg}}$ of its field of fractions).

\begin{proposition}\label{P:bicCMpow}
The ring $B_n$ is a complete local domain with maximal ideal $pB_n$ and there exists a   \homo\ $S_n\to B_n$ which is faithfully flat. In particular, $B_n$ is a big MCM algebra over  $S_n$.
\end{proposition} 
\begin{proof}
Let us write $S:=S_n$, $\bar S:=\bar S_n$, etc. Since $\bar S^+$ is a perfect ring, the theory of  strict $p$-rings as in \cite[Chapter II]{Se68}, shows that $B$ is local,  torsion-free and complete with respect to the $pB$-adic topology, and $B/pB=\bar S^+$. Suppose $ab=0$ in $B$ with $a,b\in B$ non-zero. Since $B$ is complete , whence $pB$-adically separated, we can write $a=p^ka'$ and $b'=p^mb'$ with $a',b'\notin pB$. Since $B$ has no $p$-torsion, $a'b'=0$. Taking residues modulo $pB$, we get $a'b'=0$ in $\bar S^+$, and since the latter is a domain, we get, say, $a'=0$ in $\bar S^+$, whence $a'\in pB$, contradiction.  So $B$ is a domain. Finally, let $\tilde {\tuple y}:=\tau(\tuple y)$, which therefore is a lifting of (the image of) ${\tuple y}$ in $\bar S^+$. Since $\kappa\sub\bar S^+$ we get $V=\wi\kappa\to \wi{\bar S^+}=B$, and we can extend this \homo\ by sending ${\tuple y}$ to $\tilde {\tuple y}$. Note that $(p,{\tuple y})$ is a $B$-regular sequence: indeed, $p$ is $B$-regular, and $B/pB=\bar S^+$ is known to be a big MCM $\bar S$-algebra by the work of Hochster-Huneke. Hence $B$ is a big MCM $S$-algebra, and the structure morphism $S\to B$ must therefore be flat.
\end{proof}

\subsection*{Witt transforms}  
Let $R$ be a  torsion-free, complete  local  $V$-algebra with residue field $\kappa$. Hence $R\iso S_n/I$ for some ideal $I\sub S_n$. 
Let $\bar R:=R/pR$ and consider its Witt ring $\wi{\bar R}$. 
Since $\wi{\bar R}$ is a $V$-algebra, we have a canonically defined \homo\ $j\colon S_n\to \wi{\bar R}$ sending $\tuple y$ to $\tau(\tuple y)$. 
We define the \emph{Witt transform} of $R$, denoted $\Witt R$, as the image of $S_n$ under $j$, that is to say, $\Witt R\iso S_n/\Witt I$, where $\Witt I$ is the kernel of $j$ (called the \emph{Witt transform} of $I$). From the definitions, it is clear that
\begin{equation}\label{eq:wittid}
\Witt I\sub I+pS_n.
\end{equation}

It is not hard to see (or using \Prop{P:bicCMpow}, since $S_n\to B_n$ factors through $\wi{\bar S_n}$) that $\Witt{S_n}=S_n$. It also follows from the general theory that $\wi{\bar R}$, whence its subring $\Witt R$, has no $p$-torsion,  so that $\Witt R$ is again a  torsion-free, complete  local  $V$-algebra with residue field $\kappa$, and so we have defined a functor $\Witt {(\cdot)}$ on the latter class (the functoriality follows from that of $\wi\cdot$).\footnote{Technically speaking, we defined a functor on pairs $(S_n,I)$ such that $S_n/I$ is torsion-free, since it is not clear at this point whether $\Witt R$ depends on the representation $R=S_n/I$.} 

\begin{lemma}\label{L:wittlike}
Let $R=S_n/I$ be torsion-free and $\Witt R$ its Witt-transform. Then $\Witt R$ is again torsion-free, has the same dimension as $R$, and $\bar R\iso \Witt R/p\Witt R$. In order for $R=\Witt R$, it suffices that $I\sub \Witt I$.
\end{lemma} 
\begin{proof}
Let $d:=\dim R-1$ and let $\tilde {\tuple y}:=\tau(\tuple y)$, so that $\Witt I$ is the kernel of the \homo\ $S_n\to \wi{\bar R}\colon \tuple y\to \tilde{\tuple y}$. By Noether normalization, there exists a finite, injective \homo\ $S_d\sub R$, and we may choose it so that it is the restriction of $S_n\to R$ onto the first $d$ variables (via $S_d\sub S_n)$. The \homo\ $S_d\to \Witt R$ is injective, for if $f\in S_d$ would be a non-zero element whose image in $\Witt R$ is zero, then, after factoring out a power of $p$, we may assume that $f\notin pS_d$ but $f(\tilde y)=0$. Hence, the image of $f$  is zero in $\bar R$, and this then already holds in $\bar S_d$ since it is a subring. So $f\in pS_d$, contradiction. It follows that $\Witt R$ has dimension at least $d+1$. On the other hand, the restriction of the canonical epimorphism $\wi{\bar R}\to \bar R$ to $\Witt R$ must be surjective, as the composition $S_n\to \bar S_n\to \bar R$ is. It follows that the dimension of $\Witt R$ can be at most $d+1$, proving the first assertion.

 Suppose $I\sub \Witt I$ and we need to show that the other inclusion holds as well. Let $F\in \Witt I$, so that we can write it as $F=f+pG$ with $f\in I$ and $G\in S_n$ by \eqref{eq:wittid}. By definition $F(\tilde {\tuple y})=0$ and by assumption $f(\tilde {\tuple y})=0$, whence $pG(\tilde {\tuple y})=0$. Since $\Witt R$ is torsion-free, this implies $G(\tilde {\tuple y})=0$, that is to say,  $G\in \Witt I$. Hence we showed $\Witt I=I+p\Witt I$. By Nakayama's lemma, this implies $I=\Witt I$.
\end{proof} 
\begin{remark}\label{R:wittlike}
It follows that if $R$ is \CM, then so is $\Witt R$: indeed, since $\bar R$ is equal to $\Witt R/p\Witt R$ by the lemma, and the former is \CM\ by assumption, the claim follows as $p$ is $\Witt R$-regular too. 
\end{remark}

 Since $\bar R=\Witt R/p\Witt R$, we get immediately that $\Witt {(\Witt R)}=\Witt R$. This justifies calling any  complete torsion-free $V$-algebra  of the form $\Witt R$ \emph{\Wlike}.

\begin{theorem}\label{T:Wittbcm}
Any \Wlike\ ring admits a big \CM\ algebra.
\end{theorem} 
\begin{proof}
Let $R$ be a complete, torsion-free $V$-algebra and let $d:=\dim R-1$. As in the previous proof, we may assume that the restriction of $S_n\to R$ onto the first $d$ variables yields a Noether normalization  $S_d\sub R$. This induces modulo $p$ a finite \homo\ $\bar S_d\to \bar R$, and since $\dim{\bar R}=d$, this is again injective. Let $\pr$ be a $d$-dimensional minimal prime of $\bar R$, then the composition $\bar S_d\to \bar R/\pr$ is still finite and injective. In particular, these two domains have the same absolute integral closure $\bar S_d^+$. Therefore, the composition $\bar R\to \bar R/\pr\sub \bar S_d^+$ induces a \homo\ $\wi{\bar R}\to B_d=\wi{\bar S_d^+}$, and hence by restriction, we get a \homo\ $\Witt R\to B_d$. Moreover, the composition $S_d \sub S_n\to \Witt R\to B_d$, is the restriction of the \homo\ $S_n\to B_n$ onto the first $d$ variables, and hence is flat by \Prop{P:bicCMpow}. Since  $\Witt R$ has dimension $d+1$ by \Lem{L:wittlike}, we see that $B_d$ is a big \CM\ module over it. Here is the relevant commutative diagram
$$
\xymatrix{
S_n\ar[rrr]^j\ar@{->>}[dr]&&&\wi{\bar S_n}\ar[r]\ar[ld]&B_n\\
&\Witt R\ar[r]&\wi{\bar R}\\
S_d\ar[rrr]^j\ar@{^{(}->}[uu]\ar[ur]&&&\wi{\bar S_d}\ar[uu]\ar[lu]\ar[r]&B_d\ar[uu]\\
}
$$
\end{proof}

\subsection{Proof of \Thm{T:mixtorMCM}.}
Let $V$ be a complete $p$-ring and $S:=\pow V{\tuple y}$. After a change of variables, we may take $\tuple z:=(\tuple y,p)$ as system of parameters. Let us first treat the case that   $I$ is generated by sum/differences of monomials in the variables $\tuple y$ only, and put $R:=\pow V{\tuple y}/I$. The claim follows from \Thm{T:Wittbcm} once we showed that $I=\Witt I$, and this will follow from the multiplicativity of the Teichmuller representatives: if ${\tuple y}^\alpha\pm {\tuple y}^\beta\in I$, then ${\tuple y}^\alpha=\pm {\tuple y}^\beta$ in $\bar R$, and hence $\tau({\tuple y})^\alpha=\pm \tau({\tuple y})^\beta$ in $W(\bar R)$. This shows that ${\tuple y}^\alpha\pm {\tuple y}^\beta$ lies in the kernel $\Witt I$ of the \homo\ $\pow V{\tuple y}\to W(\bar R)$, showing that $I\sub\Witt I$, and the result now follows from \Lem{L:wittlike}.

For the general purely toric case, $I$ is generated by binomials of the form $\tuple z^\alpha\pm\tuple z^\beta$. Let $S':=\pow St$ and put $\tuple x:=(\tuple y,t)$. Let $I'$ be the ideal in $S'$ generated by the  binomials $\tuple x^\alpha\pm\tuple x^\beta$, for each $\tuple z^\alpha\pm\tuple z^\beta\in I$. By the previous case, $R':=S'/I'$ admits a big MCM. Since $t-p$ is a parameter on $R'$ and $R=R'/(t-p)R'$, we are done by \Lem{L:qdefHoc}.
\qed
\begin{remark}\label{R:mixtorMCM}
The proof shows that the same result holds for any $\pmb\mu(V)$-toric ring, where $\pmb\mu(V)$ is the `circle group' consisting of all $n$-th roots of unity in $V$, since the latter are Teichmuller representatives, and even more generally, for any $\tau(\kappa)$-toric ring. 
\end{remark} 

\begin{theorem} \label{T:torNNmix}
Let $V$ be a complete $p$-ring  and $S=\pow V{y_2,\dots,y_d}$ with regular system of parameters $(p,y_2,\dots,y_d)$, let $\Gamma\sub \nat^n$ with $n\leq d$ be a full semi-group, and let $\varphi\colon\Gamma\to \nat^d$ and $\chi\colon\Gamma\to S^*$ be respectively a tame  \homo\ of semi-groups  and an $S$-character. Then the bi-partite toric ring $R:=\mathcal R(S,\Gamma,\varphi,\chi)$ admits a big MCM.
\end{theorem} 
\begin{proof}
Without loss of generality, we may assume that the residue field $\kappa$ is moreover algebraically closed. As in the previous case, we may adjoin one more variable $y_1$ and prove the result    for the (partial regular system of parameters) $\tuple y:=\rij yd$ (that is to say, the defining equations do not involve powers of $p$).
Let $\bar \chi\colon\Gamma\to \kappa$ be the composition of $\gamma$ with the residue map $S^*\to \kappa^*$. 
Since $\kappa$ is algebraically closed, $\kappa^*$ is divisible, and so $\bar\chi$ extends to a true character $\mathbb Q^n\to \kappa^*$, which we continue to denote by $\bar\chi$. Let $\epsilon_i\in\nat^n$ be the standard basis and let $l_i:=\bar\chi(\epsilon_i)$. By fullness, some positive multiple of  each $\epsilon_i$ lies in $\Gamma$; let $\gamma_i:=a_i\epsilon_i$ be the least such positive multiple. In particular,  $\bar\chi(\gamma_i)=l_i^{a_i}$ and the $i$-th row  of the matrix $\mat A_\varphi$ associated to $\varphi$ has  common denominator  dividing $a_i$. The tameness assumption implies that no $a_i$ is divisible by $p$ and hence that $\mat A_\varphi$ is defined over $\kappa$ and has maximal rank.  In particular, by Hensel's Lemma, we can   find units $g_i\in S$ such that $\chi({\gamma_i})=g_i^{a_i}$. Consider the system of equations
$$
\tuple t^{\varphi(\epsilon_i)}=g_i\quad\text{for $i=\range 1n$}
$$
 over $S$ in the $n$ unknowns $\tuple t:=\rij td$. Upon adding some dummy variables, we may assume that $n=d$. The residue in $\kappa$ of the  determinant of the Jacobian of this matrix  is of the form   $\tuple t^\beta\op{det}(\mat A_\varphi)$, for some $\beta\in\nat^d$, whence is non-zero. By the multivariate Hensel's Lemma, we can therefore find a solution $\rij zd$ of the above system in $S$ (note that each $z_i$ is necessarily a unit). Put $h_i:=z_i^{a_i}$ and $\tuple h=\rij hd$. It follows that $\tuple h^{\varphi(\gamma)}= \chi(\gamma)$, for every $\gamma\in \Gamma$. Hence we can make the change of  variables $y_i\mapsto  h_iy_i$, reducing to the case that $\chi$ is the trivial character, that is to say, the purely toric case, so that we finish off with an application of \Thm{T:mixtorMCM}.
\end{proof}

%

\providecommand{\bysame}{\leavevmode\hbox to3em{\hrulefill}\thinspace}
\providecommand{\MR}{\relax\ifhmode\unskip\space\fi MR }
\providecommand{\MRhref}[2]{%
  \href{http://www.ams.org/mathscinet-getitem?mr=#1}{#2}
}
\providecommand{\href}[2]{#2}

\end{document}